\documentclass{amsart}
\usepackage{amsthm,mathtools}
\usepackage[all]{xy}

\theoremstyle{plain}
\theoremstyle{definition}
\newtheorem{theorem}{Theorem}[section]
\newtheorem{lemma}[theorem]{Lemma}
\newtheorem{proposition}[theorem]{Proposition}
\newtheorem{corollary}[theorem]{Corollary}
\newtheorem{claim}[theorem]{Claim}

\newtheorem{example}[theorem]{Example}

\newtheorem{remark}[theorem]{Remark}
\newtheorem{question}[theorem]{Question}

\usepackage{amssymb,amsmath,tabularx,graphicx,float,placeins}
\usepackage{tikz}
\usetikzlibrary[calc,intersections,through,backgrounds,arrows,decorations.pathmorphing]

\def\Fbf{\mathbf{F}}

\def\Acal{\mathcal{A}}

\def\Ccal{\mathcal{C}}

\def\Lcal{\mathcal{L}}
\def\Mcal{\mathcal{M}}

\def\Pcal{\mathcal{P}}

\def\Tcal{\mathcal{T}}

\def\1bb{\mathbb{1}}

\def\Rbb{\mathbb{R}}
\def\Sbb{\mathbb{S}}
\def\Zbb{\mathbb{Z}}

\def\ra{\rightarrow}

\def\rsqa{\rightsquigarrow}
\def\Ra{\Rightarrow}

\def\hookra{\hookrightarrow}
\def\ov{\overline}
\def\un{\underline}
\def\wtil{\widetilde}

\def\pr{\prime}

\def\pa{\partial}
\def\setm{\setminus}
\DeclareMathOperator{\coh}{coh} 
\DeclareMathOperator{\rk}{rk}

\DeclareMathOperator{\supp}{supp}
\DeclareMathOperator{\Vol}{Vol}

\DeclareMathOperator{\sgn}{sgn}

\begin{document}

\title{Zonotopes whose cellular strings are all coherent}
\date{\today}
\author{Rob Edman, Pakawut Jiradilok, Gaku Liu, Thomas McConville}

\maketitle

\begin{abstract}
A cellular string of a polytope is a sequence of faces stacked on top of each other in a given direction. The poset of cellular strings, ordered by refinement, is known to be homotopy equivalent to a sphere. The subposet of coherent cellular strings is the face lattice of the fiber polytope, hence is homeomorphic to a sphere. In some special cases, every cellular string is coherent. Such polytopes are said to be all-coherent. We give a complete classification of zonotopes with the all-coherence property in terms of their oriented matroid structure. Although the face lattice of the fiber polytope in this case is not an oriented matroid invariant, we prove that the all-coherence property is invariant.
\end{abstract}

\section{Introduction}

We consider the problem of realizing spaces of polytopal subdivisions as the boundary complex of a polytope. For example, the set of triangulations of a convex polygon form the vertices of a polytope known as the associahedron. The facets of an associahedron correspond to diagonals where a facet and vertex are incident exactly when the triangulation includes the diagonal. In this situation, we are given the combinatorics of the polytope, namely the face lattice, and we need to provide a geometric realization.

Among the many polytopal realizations of the associahedron (see \cite{ceballos.santos.ziegler2015many} for a survey), the secondary polytope construction of Gelfand, Kapranov, and Zelevinsky is particularly elegant \cite[Chapter 7]{gelfand.kapranov.zelevinsky2008discriminants}. Given a finite set of points $\Acal$ in $\Rbb^d$, the \emph{secondary polytope} $\Sigma(\Acal)$ is the Newton polytope of the $\Acal$-determinant. Its vertices correspond to \emph{regular triangulations} of $\Acal$, where a triangulation $\Tcal$ is \emph{regular} (or \emph{coherent}) if there exists a function $f:\Acal\ra\Rbb$ such that $\Tcal$ is the set of lower faces of the convex hull of $\{(x,f(x)):\ x\in\Acal\}\subseteq\Rbb^{d+1}$. If $\Acal$ is the set of vertices of a convex polygon, then every triangulation is regular, so the secondary polytope realizes the associahedron.

Given a linear surjection of polytopes $\pi:P\ra Q$, the \emph{Baues poset} $\omega(P,Q)$ is the set of polytopal subdivisions of $Q$ by images of faces of $P$, ordered by refinement; see Section~\ref{subsec_fiber} for a more precise definition. For example, when $P$ is a simplex, the maximally refined subdivisions of $Q$ in $\omega(P,Q)$ are a family of triangulations with specified vertices. For generic functions $\pi$, Baues observed that if $P$ is a $d$-simplex and $Q$ is $1$-dimensional, then $\omega(P,Q)$ is the face lattice of a $(d-1)$-cube \cite{baues:geometry}. If $P$ is a $d$-cube and $Q$ is $1$-dimensional, then $\omega(P,Q)$ is the face lattice of a $(d-1)$-dimensional permutahedron. When $P$ is a $d$-dimensional permutahedron, the maximally refined subdivisions of $Q$ may be identified with reduced words for the longest element of a type $A_d$ Coxeter system. Baues noticed that the order complex of $\omega(P,Q)$ is not always a simplicial sphere, but he conjectured that it is homotopy-equivalent to a sphere. This was proved for any polytopes $P,Q$ with $\dim Q=1$ by Billera, Kapranov, and Sturmfels \cite{billera.kapranov.sturmfels:cellular}. We refer to the survey article by Reiner for many other results of this type \cite{reiner:baues}.

As with triangulations, there is a distinguished subset of polytopal subdivisions in $\omega(P,Q)$ called \emph{coherent} subdivisions. Generalizing the secondary polytope construction, the \emph{fiber polytope} introduced by Billera and Sturmfels is a polytope of dimension $\dim P-\dim Q$ whose faces correspond to coherent subdivisions of $Q$ \cite{billera.sturmfels:fiber}. Hence, the Baues poset $\omega(P,Q)$ contains a canonical spherical subspace of dimension $\dim P-\dim Q-1$, and Baues' original question is whether this subspace is homotopy-equivalent to the full poset $\omega(P,Q)$. We consider the stronger question:

\begin{question}\label{ques_main}
For which projections $f:P\ra Q$ is $\omega(P,Q)$ isomorphic to the face lattice of a polytope of dimension $\dim P-\dim Q$?
\end{question}

Using the fiber polytope construction, the poset $\omega(P,Q)$ is the face lattice of a $(\dim P-\dim Q)$-dimensional polytope exactly when every element of $\omega(P,Q)$ is coherent; see Section~\ref{sec_coherence_poly} for details. Hence, we may rephrase Question~\ref{ques_main} as asking whether every subdivision of $Q$ by faces of $P$ is coherent. If $\pi:P\ra Q$ has this property, we say the projection is \emph{all-coherent}.

In this work, we examine the special case of Question~\ref{ques_main} where $P$ is a zonotope and $Q$ is 1-dimensional. A \emph{zonotope with $n$ zones} is the image of a $n$-dimensional cube under a linear map. To any zonotope with a distinguished linear functional, one may associate an (acyclic) \emph{oriented matroid}, an abstract combinatorial object that records the faces of the zonotope. Some background on oriented matroids is given in Section~\ref{sec_OM}. Given $P$ and $Q$ as above, it is known that the facial structure of the fiber polytope of $\pi:P\ra Q$ is not an oriented matroid invariant. That is, there exists a zonotope with two distinct linear functionals defining the same oriented matroid whose fiber polytopes are not combinatorially equivalent; see Example~\ref{example_zonogon}. However, we prove that the all-coherence property \emph{is} invariant.

Our main result is a characterization of the all-coherence property for the set of cellular strings of a zonotope.

\begin{theorem}\label{thm_main}
Let $Z$ be a zonotope with a generic linear functional $\pi$. The following are equivalent.

\begin{enumerate}
\item\label{thm_main_1} The pair $(Z,\pi)$ is all-coherent.
\item\label{thm_main_2} The Baues poset $\omega(Z,\pi)$ is the face lattice of a polytope.
\item\label{thm_main_3} The order complex $\Delta(\omega(Z,\pi))$ is a simplicial sphere.
\item\label{thm_main_4} Every monotone path is coherent.
\item\label{thm_main_6} The acyclic oriented matroid $\Mcal$ associated to $(Z,\pi)$ is in the list given in Section~\ref{sec_classification}.
\end{enumerate}

\end{theorem}

The equivalence of (\ref{thm_main_1})-(\ref{thm_main_3}) is proved for all polytopes in Corollary~\ref{cor_all_coh}. Since the order complex $\Delta(\omega(Z,\pi))$ only depends on the oriented matroid corresponding to $(Z,\pi)$, we may take Statement (\ref{thm_main_3}) as the definition of the all-coherence property for an oriented matroid. We prove the equivalence of (\ref{thm_main_1}) and (\ref{thm_main_4}) in Corollary~\ref{cor_mono_path}. Our proof of this equivalence works for zonotopes only. Finally, we give a complete classification of zonotopes with the all-coherence property in Section~\ref{sec_classification} using the techniques developed in Section~\ref{sec_coherence_lemma}.

The all-coherence property was previously studied for projections $P\ra Q$ where $P$ is a cube and $Q$ is 2-dimensional \cite{edelman.reiner:1996free} or $P$ and $Q$ are cyclic polytopes \cite{athanasiadis.deloera.reiner.santos:2000fiber}. It was observed that in many cases of an all-coherent projection, there is a ``nice'' formula for the number of generic tilings. In Section~\ref{subsec_allcoh_family}, we demonstrate a family of zonotopes for which every cellular string is coherent. As in the preceding literature, the monotone paths of these zonotopes admit a nice enumeration.

This paper continues the work by the first author \cite{edman:2015diameter}, which contained several techniques to detect incoherent monotone paths and had a partial classification of all-coherent zonotopes. However, our proof of the full classification primarily relies on Proposition~\ref{prop_graded}, which is new to this paper.

\section{Universal coherence for polytopes}\label{sec_polytope}

In this section, we recall coherence and the Baues poset of polyhedral subdivisions induced by a projection of polytopes. In Section~\ref{sec_coherence_poly}, we prove that every element of the Baues poset is coherent exactly when its order complex is homeomorphic to a sphere of a certain dimension.

\subsection{Fiber polytope}\label{subsec_fiber}

Given a polytope $P$ and linear functional $\psi$, let $P^\psi$ be the face
$$P^\psi=\{x\in P:\ \psi(x)=\min_{y\in P}\psi(y)\}.$$

That is, $P^{\psi}$ is the face of $P$ at which $\psi$ is minimized. The \emph{normal fan} of $P\subseteq\Rbb^d$ is a subdivision of $(\Rbb^d)^*$ into relatively open cones where $\psi_1$ and $\psi_2$ lie in the same cone if $P^{\psi_1}=P^{\psi_2}$. Let $\pi:P\ra Q$ be a surjective (affine) linear map of polytopes. A collection $\Delta$ of faces of the boundary complex of $P$ is called a \emph{$\pi$-induced subdivision} of $Q$ if
\begin{itemize}
\item $\{\pi(F):\ F\in\Delta\}$ is a polyhedral complex subdividing $Q$ with all $\pi(F)$ distinct, and 
\item if $\pi(F)\subseteq\pi(F^{\pr})$ then $F=F^{\pr}\cap\pi^{-1}(\pi(F))$.
\end{itemize}

For any $\pi:P\ra Q$, there is always the \emph{trivial subdivision} $\Delta=\{P\}$. A $\pi$-induced subdivision $\Delta$ is \emph{coherent} if there exists a linear functional $\psi:P\ra\Rbb$ such that for $x\in Q$ and $F\in\Delta$ with $x\in\pi(F)$,
$$\pi^{-1}(x)^\psi=F\cap\pi^{-1}(x).$$
Any linear functional $\psi$ defines a unique coherent subdivision $\Delta^\psi$. For example, $\Delta^0$ is the trivial subdivision.

The \emph{Baues poset} $\omega(P,Q)$ is the set of $\pi$-induced subdivisions, partially ordered by refinement: $\Delta\leq\Delta^{\pr}$ if for any face $F$ in $\Delta$ there exists $F^{\pr}$ in $\Delta^{\pr}$ such that $F\subseteq F^{\pr}$. The subposet of coherent subdivisions is denoted $\omega_{\coh}(P,Q)$.

The \emph{fiber polytope} $\Sigma(P,Q)$ was defined by Billera and Sturmfels as the Minkowski integral $\frac{1}{\Vol(Q)}\int_Q\pi^{-1}(x)dx$; see \cite{billera.sturmfels:fiber} for details. This polytope lives in the ambient space of $P$ and is of dimension $\dim P-\dim Q$. The faces of the fiber polytope encode coherent subdivisions of $Q$ in the following sense. For linear functionals $\psi_1,\psi_2:P\ra\Rbb$, $\Delta^{\psi_1}=\Delta^{\psi_2}$ if and only if $\psi_1$ and $\psi_2$ lie in the relative interior of the same cone of the normal fan of $\Sigma(P,Q)$. For the most part, we will only need results about the normal fan of a fiber polytope rather than the polytope itself.

\subsection{All-coherence property}\label{sec_coherence_poly}

Given a poset $X$, the \emph{order complex} $\Delta(X)$ is the abstract simplicial complex whose faces are chains $x_0<\cdots<x_d$ of $X$. If $X$ has a unique maximum or minimum element, we let $\ov{X}$ be the same poset with these elements removed. In particular, $\ov{\omega}(P,Q)$ is the poset of \emph{nontrivial} subdivisions of $Q$ by faces of $P$.

We say that a projection $\pi:P\ra Q$ is \emph{all-coherent} if every $\pi$-induced subdivision of $Q$ is coherent. The following theorem gives a useful characterization of the all-coherence property.

\begin{theorem}\label{thm_all_coh}
Given a projection $\pi:P\ra Q$, the following statements are equivalent.
\begin{enumerate}
\item\label{thm_all_coh_1} The map $\pi$ is all-coherent.
\item\label{thm_all_coh_2} The Baues poset $\omega(P,Q)$ is isomorphic to the face lattice of a polytope of dimension $\dim P-\dim Q$.
\item\label{thm_all_coh_3} The order complex $\Delta(\ov{\omega}(P,Q))$ is a simplicial sphere of dimension $\dim P-\dim Q-1$.
\end{enumerate}
\end{theorem}

\begin{proof}
We prove (\ref{thm_all_coh_1})$\Ra$(\ref{thm_all_coh_2})$\Ra$(\ref{thm_all_coh_3})$\Ra$(\ref{thm_all_coh_1}).

If $P\ra Q$ is all-coherent, then $\omega(P,Q)=\omega_{\coh}(P,Q)$. The latter poset is isomorphic to the face lattice of the fiber polytope $\Sigma(P,Q)$, which is a polytope of dimension $\dim P-\dim Q$.

If $\omega(P,Q)$ is the face lattice of a polytope $R$, then its proper part $\ov{\omega}(P,Q)$ is the face poset of its boundary complex $\pa R$. The order complex $\Delta(\ov{\omega}(P,Q))$ is realized by the barycentric subdivision of $\pa R$, so it is a simplicial sphere. If $\dim R=\dim P-\dim Q$, then its boundary complex has dimension $\dim P-\dim Q-1$.

Now assume $\Delta(\ov{\omega}(P,Q))$ is a simplicial sphere of dimension $d=\dim P-\dim Q-1$. The inclusion $\omega_{\coh}(P,Q)\hookra\omega(P,Q)$ induces an injective simplicial map
$$\Delta(\ov{\omega}_{\coh}(P,Q))\hookra\Delta(\ov{\omega}(P,Q)).$$
The geometric realization of this map is a continuous map $\Sbb^d\ra\Sbb^d$ which is a homeomorphism onto its image. If it is not surjective, then we obtain a subspace of $\Rbb^d$ which is homeomorphic to $\Sbb^d$. This is well-known to be impossible, e.g. by an application of the Borsuk-Ulam Theorem; see \cite[Theorem 2.1.1]{matousek2008borsuk}. Hence, $\omega(P,Q)=\omega_{\coh}(P,Q)$ holds, which means that $P\ra Q$ is all-coherent.
\end{proof}

If $\dim Q=1$, we can drop the dimension condition. This proves the equivalence of Theorem~\ref{thm_main}(\ref{thm_main_1})-(\ref{thm_main_3}).

\begin{corollary}\label{cor_all_coh}
If $\dim Q=1$, then the following are equivalent.
\begin{enumerate}
\item\label{cor_all_coh_1} The map $\pi:P\ra Q$ is all-coherent.
\item\label{cor_all_coh_2} The Baues poset $\omega(P,Q)$ is the face lattice of a polytope.
\item\label{cor_all_coh_3} The order complex $\Delta(\ov{\omega}(P,Q))$ is a simplicial sphere.
\end{enumerate}
\end{corollary}

\begin{proof}
The proof of (\ref{cor_all_coh_1})$\Ra$(\ref{cor_all_coh_2})$\Ra$(\ref{cor_all_coh_3}) is immediate. It remains to show (\ref{cor_all_coh_3}) implies (\ref{cor_all_coh_1}).

Suppose $\dim Q=1$ and $\Delta(\ov{\omega}(P,Q))$ is a simplicial sphere. Then \cite[Theorem 1.2]{billera.kapranov.sturmfels:cellular} states that $\Delta(\ov{\omega}(P,Q))$ is homotopy equivalent to a sphere of dimension $\dim P-2$. Hence, $\Delta(\ov{\omega}(P,Q))$ is homeomorphic to $\Sbb^{\dim P-2}$. By Theorem~\ref{thm_all_coh}, it follows that $\pi$ is all-coherent.
\end{proof}

\section{Oriented matroids}\label{sec_OM}

In this section, we develop some notation for oriented matroids. For more background on oriented matroids, we recommend \cite{bjorner.lasVergnas.ea:oriented}.

\subsection{Hyperplane arrangements}\label{subsec_hyparr}

The normal fan of a zonotope is the set of cones cut out by a real, central hyperplane arrangement. In this section, we recall some notation for hyperplane arrangements.

A \emph{real hyperplane arrangement} (or \emph{arrangement}) is a finite set of hyperplanes in a real vector space $V$. We will assume that the arrangement is \emph{central}, which means that every hyperplane contains the origin. Each hyperplane partitions $V$ into two open half-spaces and the hyperplane itself. We typically assume that an arrangement comes with an \emph{orientation}; that is, for each hyperplane $H$, we set $H^0=H$ and assign the label $H^+$ and $H^-$ to the two open half-spaces defined by $H$. A \emph{face} of an arrangement $\Acal$ is any nonempty cone of the form $\bigcap_{H\in\Acal}H^{x(H)}$ where $x\in\{0,+,-\}^{\Acal}$. Maximal faces are called \emph{chambers}. A \emph{sign vector} is an element of $\{0,+,-\}^E$ for some finite set $E$. A sign vector $x\in\{0,+,-\}^{\Acal}$ is a \emph{covector} of $\Acal$ if there is a face $F$ of $\Acal$ such that $F = \bigcap_{H\in\Acal}H^{x(H)}$. The set of covectors of $\Acal$ will be denoted $\Lcal(\Acal)$.

We say a zonotope $Z=Z(\mathbf{\rho})$ is generated by a set of vectors $\mathbf{\rho}=\{\rho_1,\ldots,\rho_n\}\subseteq V$ if it is (an affine translation of) the image of $[0,1]^n$ under the matrix with columns vectors $\rho_1,\ldots,\rho_n$. The normal fan of $Z$ is the hyperplane arrangement $\Acal=\{H_1,\ldots,H_n\}$ in $V^*$ where $H_i$ is the set of linear functionals annihilating $\rho_i$. In this situation, we say that $Z$ is dual to $\Acal$.

\begin{example}\label{example_zonogon}
Set $I=[0,1]$ and let $\pi:I^6\ra Z$ be a generic projection from the standard $6$-dimensional cube to a 2-dimensional zonotope $Z$ with $6$ zones. We may assume that $\pi$ takes the form
$$\left(\begin{matrix}1 & 1 & 1 & 1 & 1 & 1\\a_1 & a_2 & a_3 & a_4 & a_5 & a_6\end{matrix}\right)$$
for some $a_1<\cdots<a_6$. A linear map $\hat{\pi}:\Rbb^6\ra\Rbb^3$ that is a factor of $\pi$ takes the form
$$\left(\begin{matrix}1 & 1 & 1 & 1 & 1 & 1\\a_1 & a_2 & a_3 & a_4 & a_5 & a_6\\b_1 & b_2 & b_3 & b_4 & b_5 & b_6\end{matrix}\right)$$
for some $b_1,\ldots,b_6\in\Rbb$. Let $\hat{Z}$ be the image of $I^6$ under $\hat{\pi}$. For $1\leq i<j<k\leq 6$, let $\Delta_{ijk}^{\un{a}}$ denote the minor
$$a_ib_j-a_ib_k+a_jb_k-a_jb_i+a_kb_i-a_kb_j.$$
The sequence of signs $(\sgn\Delta_{ijk}^{\un{a}})_{1\leq i<j<k\leq 6}$ is the \emph{chirotope} of $\hat{Z}$, and it completely determines the facial structure of $\hat{Z}$. As a function of $\un{b}$, the minor $\Delta_{ijk}^{\un{a}}$ is a linear form on $\Rbb^6$. Hence, the normal fan of the fiber polytope associated to the projection $\pi$ is the set of cones cut out by the hyperplane arrangement
$$\Acal^{\un{a}}=\{\ker\Delta_{ijk}^{\un{a}}:\ 1\leq i<j<k\leq 6\},$$
which is known as a \emph{discriminantal arrangement} (cf. \cite{bayer.brandt:1997discriminantal}).

The Baues poset $\omega(I^6,Z)$ is the collection of subdivisions of $Z$ by zonogons, each of which is the Minkowski sum of several edges of $Z$. In particular, the maximally refined subdivisions of $Z$ are by rhombi. The coherent rhombic tilings correspond to the chambers of $\Acal^{\un{a}}$. There are a total of 904 distinct rhombic tilings of $Z$. However, there is no choice of vector $\un{a}$ for which all 904 tilings are coherent. A simple way to see this is that the (proper) subposet of $\omega(I^6,Z)$ of tilings refining the zonotopal subdivision in Figure~\ref{fig_zonotile} is a rank 4 Boolean lattice, but the subposet of coherent subdivisions is the face lattice of a 4-dimensional polytope. Moreover, the number of chambers of $\Acal^{\un{a}}$ may vary for different choices of $\un{a}$. For example, one can check that $\Acal^{(1,2,3,4,5,6)}$ has $888$ chambers, whereas $\Acal^{(1,2,3,4,5,7)}$ has $892$ chambers. Sturmfels \cite{sturmfels:1991letter} classified the combinatorial types of discriminantal arrangements of this form, showing that it has at least $876$ and at most $892$ chambers.

The all-coherence property for a projection of the form $\pi:I^n\ra Z,\ Z\subseteq\Rbb^2$ was completely classified by Edelman and Reiner \cite{edelman.reiner:1996free}. In that paper, the authors observed that the number of rhombus tilings of $Z$ admits a nice product formula when they are all-coherent, as given by MacMahon when $Z$ is a hexagon \cite{macmahon:2001combinatory} or Elnitsky when $Z$ is an octogon \cite{elnitsky:1997rhombic}. Interestingly, many (but not all) of the discriminantal arrangements are free when the projection $\pi$ has the all-coherence property. However, no direct connection between freeness and all-coherence is known. Bailey \cite{bailey:1999coherence} observed a similar connection between all-coherence of the zonotopal tilings of a $3$-dimensional zonotope $Z$ and freeness of the corresponding discriminantal arrangement.

\end{example}

\begin{figure}
\centering
\includegraphics[scale=1.2]{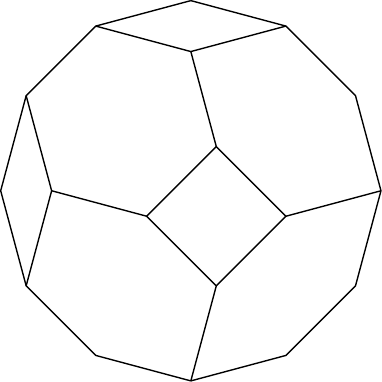}
\caption{\label{fig_zonotile}A zonogon tiling}
\end{figure}

\subsection{Cellular strings}\label{subsec_cellular_strings}

Fix an arrangement $\Acal$ of hyperplanes in $\Rbb^n$. Designate one of the chambers of $\Acal$ to be the \emph{fundamental chamber} $c_0$. Given two chambers $c$ and $c'$ of $\Acal$, let $S(c,c')$ denote the set of hyperplanes of $\Acal$ for which $c$ and $c'$ lie on opposite sides of the hyperplane. A \emph{gallery} is a sequence of chambers $c_0,\ldots,c_l$ such that $|S(c_{i-1},c_i)|=1$ for all $i$. This gallery may be drawn as a path in $\Rbb^n$ connecting generic points in each chamber. It is \emph{reduced} if $|S(c_0,c_l)|=l$, which means that it may be represented by a path that crosses each hyperplane at most once. For our purposes, a gallery is always assumed to be reduced and connecting a pair of antipodal chambers $c_0,\ -c_0$. Such a gallery defines a total order on the arrangement $\Acal$ where $H<H^{\pr}$ if $H$ is crossed before $H^{\pr}$. If $Z$ is a zonotope dual to $\Acal$, then a gallery on $\Acal$ corresponds to a monotone path on $Z$. For this correspondence, one chooses a linear functional $\pi$ that is minimized on $Z$ at $v_0$, the vertex whose normal cone is $c_0$.

The cellular strings of $\Acal$ may be defined in several equivalent ways, e.g. as sequences of faces, covectors (Section~\ref{subsec:covector}), or chambers. Given a fundamental chamber $c_0$, we orient each $H\in\Acal$ so that $c_0$ is on the negative side of $H$. We define a \emph{cellular string} as a sequence of faces $(F_1,\ldots,F_l)$ of $\Acal$ such that for each $H\in\Acal$, there exists $i\in[l]$ such that $F_i$ is supported by $H$, $F_j\subseteq H^-$ if $j<i$ and $F_j\subseteq H^+$ if $j>i$. In particular, a cellular string determines an ordered set partition $(\Acal_1,\ldots,\Acal_l)$ of the set $\Acal$.

In terms of a zonotope $Z$ with generic functional $\pi$, a cellular string is a sequence of faces $(F_1,\ldots,F_l)$ where $F_1$ contains the bottom vertex, $F_l$ contains the top vertex, and adjacent faces $F_i,F_{i+1}$ meet at a vertex $v_i$ such that $\pi(x)\leq\pi(v_i)\leq\pi(y)$ for $x\in F_i,\ y\in F_{i+1}$. We note that this definition of a cellular string agrees with that of a $\pi$-induced subdivision given in Section~\ref{subsec_fiber}.

\begin{lemma}\label{lem_coh_cell_str}
Given a real vector space $V$, let $\Acal$ be a set of hyperplanes in $V^*$ and let $\pi$ be a generic point in $V^*$. For $H\in\Acal$, let $v_H$ be a vector in $V$ such that $\langle v_H,H\rangle=0$ and $\pi(v_H)>0$, and let $Z$ be the zonotope generated by $\{v_H:\ H\in\Acal\}$. Then an ordered set partition $(\Acal_1,\ldots,\Acal_l)$ of $\Acal$ defines a $\pi$-coherent cellular string of $Z$ if and only if there exists $\psi\in V^*$ such that
$$\frac{\psi(v_H)}{\pi(v_H)}=\frac{\psi(v_{H^{\pr}})}{\pi(v_{H^{\pr}})}$$
for $H,H^{\pr}$ in the same block, and
$$\frac{\psi(v_H)}{\pi(v_H)}<\frac{\psi(v_{H^{\pr}})}{\pi(v_{H^{\pr}})}$$
if $H\in\Acal_i,\ H^{\pr}\in\Acal_j,\ i<j$.
\end{lemma}

\begin{proof}
First assume that the ordered set partition $(\Acal_1,\ldots,\Acal_l)$ of $\Acal$ defines a $\pi$-coherent cellular string. Let $\Fbf=(F_1,\ldots,F_l)$ denote the sequence of faces of $Z$ and $(x_1,\ldots,x_l)$ the sequence of covectors associated to this cellular string. Since this string is $\pi$-coherent, there exists a linear functional $\psi$ such that $\Delta^{\psi}=\Fbf$.

For an index $i$, let $p_i=\sum_{j:\ j<i}\sum_{H\in\Acal_j}v_H$. The face $F_i$ is the Minkowski sum of the point $p_i$ with the image of $[0,1]^{|\Acal_i|}$ under a matrix with column set $\{v_H:\ H\in\Acal_i\}$.

Fix some $q\in\Rbb$ such that $q>\pi(p_i)$ and $q<\pi(p_i+v_H)$ for $H\in\Acal$. By the coherence assumption, the face $\pi^{-1}(q)^{\psi}$ is equal to $\pi^{-1}(q)\cap F_i$. Fix $H,H^{\pr}\in\Acal$, and set $v=v_H,\ v^{\pr}=v_{H^{\pr}}$. Pick $\lambda,\lambda^{\pr}\in[0,1]$ such that $\pi(p_i+\lambda v)=q=\pi(p_i+\lambda^{\pr}v^{\pr})$. Then $\lambda\pi(v)=\lambda^{\pr}\pi(v^{\pr})$.

Assume $H$ and $H^{\pr}$ are both in $\Acal_i$. Then $\psi(p_i+\lambda v)=\psi(v_i+\lambda^{\pr}v^{\pr})$, so $\lambda\psi(v)=\lambda^{\pr}\psi(v^{\pr})$. Consequently, $\frac{\psi(v)}{\pi(v)}=\frac{\psi(v^{\pr})}{\pi(v^{\pr})}$.

On the other hand, suppose $H\in\Acal_i$ and $H^{\pr}\in\Acal_j$ with $i<j$. Then $p_i+\lambda^{\pr}v^{\pr}$ is in $Z$ but not in $F_i$. Hence, $\psi(p_i+\lambda v)<\psi(p_i+\lambda^{\pr}v^{\pr})$ holds, which implies $\frac{\psi(v)}{\pi(v)}<\frac{\psi(v^{\pr})}{\pi(v^{\pr})}$.\\

Now let $(\Acal_1,\ldots,\Acal_l)$ be an ordered set partition of $\Acal$ and assume that there exists $\psi\in V^*$ such that
$$\frac{\psi(v_H)}{\pi(v_H)}=\frac{\psi(v_{H^{\pr}})}{\pi(v_{H^{\pr}})}$$
for $H,H^{\pr}$ in the same block, and
$$\frac{\psi(v_H)}{\pi(v_H)}<\frac{\psi(v_{H^{\pr}})}{\pi(v_{H^{\pr}})}$$
if $H\in\Acal_i,\ H^{\pr}\in\Acal_j,\ i<j$.

For each $i$, let $\lambda_i$ denote the ratio $\frac{\psi(v_H)}{\pi(v_H)}$ for some $H\in\Acal_i$. Let $p_i=\sum_{j:\ j<i}\sum_{H\in\Acal_j}v_H$. Let $F_i$ be the zonotope generated by $\{v_H:\ H\in\Acal_i\}$, translated by $p_i$.

We claim that $F_i$ is a face of $Z$, namely that $F_i=Z^{\psi_i}$ where $\psi_i=\psi-\lambda_i\pi$. For $H\in\Acal$,

$$\psi_i(v_H)=\psi(v_H)-\lambda_i\pi(v_H)\begin{cases}=0\ \mbox{if }H\in\Acal_i\\<0\ \mbox{if }H\in\Acal_j,\ j<i\\>0\ \mbox{if }H\in\Acal_j,\ j>i\end{cases}.$$

Using the case $H\in\Acal_i$, we see that $\psi_i$ takes a common value $\psi_i(p_i)$ on the zonotope $F_i$. The zonotope $Z$ is the set of points
$$Z=\{\sum_{H\in\Acal}\mu_H v_H:\ (\mu_H)\in[0,1]^\Acal\}.$$

For $p\in Z$, there exists a representation $p-p_i=\sum_{H\in\Acal}\mu_H v_H$ where $\mu_H\leq 0$ for $H\in\Acal_j,\ j<i$ and $\mu_H\geq 0$ for $H\in\Acal_j,\ j>i$. By the previous calculation, we have $\psi_i(p-p_i)\geq 0$. Hence, $\psi_i$ is minimized on $Z$ at $F_i$.

The image $\pi(F_i)$ is the interval $[\pi(p_i),\pi(p_{i+1})]$. Hence, the images of the faces $F_i$ partition the image of $Z$ into intervals. For $q\in\pi(F_i)$, $\psi_i$ is minimized on $\pi^{-1}(q)$ at $F_i\cap\pi^{-1}(q)$. But $\psi_i=\psi-\lambda_i\pi$ and $\pi$ is constant on $\pi^{-1}(q)$, so $\psi$ is minimized on $\pi^{-1}(q)$ at $F_i\cap\pi^{-1}(q)$ as well. Therefore, $(F_1,\ldots,F_l)$ is a $\pi$-coherent cellular string induced by $\psi$.

\end{proof}

Since the set of cellular strings of $\Acal$ starting from $c_0$ does not depend on the choice of $\pi\in c_0$, we may define the pair $(\Acal,c_0)$ to be \emph{all-coherent} if there exists some $\pi$ in the interior of $c_0$ such that $\pi:Z\ra\Rbb^1$ is all-coherent. If $\pi$ is replaced by some other $\pi^{\pr}$ in $c_0$, it would still be all-coherent by Corollary~\ref{cor_all_coh}.

Given a gallery $\gamma$ in $\Acal$ and an intersection subspace $X$, there is an induced gallery $\gamma_X$ on the localized arrangement $\Acal_X$ where $\gamma_X$ crosses elements of $\Acal_X$ in the same order as $\gamma$. If $X\in L_2$, then there are two possibilities for $\gamma_X$. If $\gamma$ and $\gamma^{\pr}$ are galleries, the \emph{$L_2$-separation set} $L_2(\gamma,\gamma^{\pr})$ is the set of $X\in L_2$ such that $\gamma_X\neq\gamma^{\pr}_X$. The set of galleries forms a graph such that $\gamma$ and $\gamma^{\pr}$ are adjacent if $|L_2(\gamma,\gamma^{\pr})|=1$.

\begin{remark}
The graph of galleries is known to be connected \cite{cordovil.moreira:1993homotopy}. More recently, the diameter of this graph was investigated by Reiner and Roichman \cite{reiner.roichman:2013diameter}. In their paper, they observe that the graph-theoretic distance between two galleries $\gamma$ and $\gamma^{\pr}$ is bounded below by $|L_2(\gamma,\gamma^{\pr})|$. On the other hand, if $(\Acal,c_0)$ is all-coherent, then there exists a path between $\gamma$ and $\gamma^{\pr}$ of length equal to $|L_2(\gamma,\gamma^{\pr})|$. We conjecture that the converse statement holds: If the distance between any reduced galleries $\gamma$ and $\gamma^{\pr}$ is given by $|L_2(\gamma,\gamma^{\pr})|$, then $(\Acal,c_0)$ is all-coherent.
\end{remark}

\subsection{An all-coherent family}\label{subsec_allcoh_family}

In this section, we consider an explicit 2-parameter family of zonotopes $Z$ with linear functional $\pi$ such that $(Z,\pi)$ is all-coherent. We present this example in detail as it will be useful for the classification in Section~\ref{sec_classification}.

Fix positive integers $r,m$. Let $e_1,\ldots,e_r$ be the elementary basis vectors in $\Rbb^r$, and define vectors $a_1,\ldots,a_m$ where $a_i=ie_r+\sum_{j=1}^re_j$. Let $Z$ be the zonotope generated by $E=\{a_1,\ldots,a_m,e_1,\ldots,e_r\}$. Let $\pi:\Rbb^r\ra\Rbb$ be the sum of coordinates function. In particular, $\pi(e_i)=1$ and $\pi(a_j)=j+r$.

We claim that $(Z,\pi)$ is all-coherent. To prove this, we first determine some necessary conditions for a permutation of $E$ to define a monotone path. Then we show that each such permutation may be realized as a coherent monotone path by a suitable linear functional. Finally, we invoke Lemma~\ref{lem_coh_galleries} to complete the proof.

Let $p$ be a monotone path on $Z$, and let $\prec$ be the induced total order on $E$. For $1\leq k<l\leq m$, we have a relation $(l-k)e_r-a_l+a_k=0$, so either $e_r\prec a_l\prec a_k$ or $a_k\prec a_l\prec e_r$ holds. Combining these relations, we have either $e_r\prec a_m\prec \cdots\prec a_1$ or $a_1\prec \cdots\prec a_m\prec e_r$. By the relation $a_i=ie_r+\sum_{j=1}^re_j$, the maximum and minimum elements under $\prec$ must be elementary basis vectors.

Now assume $\prec$ is a total ordering on $E$ whose maximum and minimum elements are elementary basis vectors, and either $e_r\prec a_m\prec \cdots\prec a_1$ or $a_1\prec \cdots\prec a_m\prec e_r$ holds. We construct a linear functional $\psi$ realizing $\prec$ as a coherent monotone path. To do so, we distinguish three cases:
\begin{enumerate}
\item\label{ex_coh_case1} $e_r$ is maximal with respect to $\prec$,
\item\label{ex_coh_case2} $e_r$ is minimal with repsect to $\prec$, or
\item\label{ex_coh_case3} $e_r$ is neither minimal nor maximal.
\end{enumerate}

We let $\psi_j=\psi(e_j)$ and $\alpha=\sum_{j=1}^r\psi_j$. We observe that $\psi(e_j)/\pi(e_j)=\psi_j$ and $\psi(a_i)/\pi(a_i)=(\alpha+i\psi_r)/(i+r)$. We assume throughout that $\psi$ is chosen so that if $e_i\prec e_j$ then $\psi_i<\psi_j$. We impose additional constraints on $\psi$ separately for the three cases above.

Suppose we are in Case~\ref{ex_coh_case1}, and let $e_k$ be the minimum element under $\prec$. Set $\psi_r=1$, and for $j\in[r-1]\setm\{k\}$, let $\psi_j$ be a positive number such that $0<\psi_j<1$ and
\begin{itemize}
\item if $a_i\prec e_j$ then $\frac{i}{i+r}<\psi_j$, and
\item if $e_j\prec a_i$ then $\psi_j<\frac{i}{i+r}$.
\end{itemize}

Finally, choose $\psi_k<0$ such that $\alpha=0$. Then $\psi(a_i)/\pi(a_i)=\frac{i}{i+r}$. By the preceding conditions, we conclude that $\prec$ is realized as a coherent monotone path by $\psi$.

Case~\ref{ex_coh_case2} follows by similar reasoning as Case~\ref{ex_coh_case1}. Assume we are in Case~\ref{ex_coh_case3}, and let $e_k$ and $e_l$ be the minimum and maximum elements of $E$ under $\prec$, respectively. Furthermore, assume that $a_1\prec\cdots\prec a_m\prec e_r$ holds. Set $\psi_r=0$, and for $j\in[r-1]\setm\{k,l\}$, let $\psi_j$ be a real number such that
\begin{itemize}
\item if $a_i\prec e_j$ then $\frac{-1}{i+r}<\psi_j$, and
\item if $e_j\prec a_i$ then $\psi_j<\frac{-1}{i+r}$.
\end{itemize}

Finally, choose $\psi_k<0,\ \psi_l>0$ such that $\psi_k<\psi_i<\psi_l$ for all $i$ and $\alpha=-1$. Then $\psi(a_i)/\pi(a_i)=\frac{-1}{i+r}$, and we again conclude that $\prec$ is realized as a coherent monotone path by $\psi$. A similar argument exists when $e_r\prec a_m\prec\cdots\prec a_1$ holds. In that case, we fix $\alpha=1$ and choose $\psi$ accordingly.

We have now proven that every monotone path is coherent. By Lemma~\ref{lem_coh_galleries}, this implies that every cellular string is coherent.

In our proof of all-coherence of $(Z,\pi)$, we identified the monotone paths on $Z$ with total orderings of $E$ such that

\begin{enumerate}\label{enum_mps}
\item\label{enum_mps_1} the maximum and minimum elements of $E$ are equal to $e_k,e_l$ for some $k,l$, and
\item\label{enum_mps_2} either $e_r\prec a_m\prec\cdots\prec a_1$ or $a_1\prec\cdots\prec a_m\prec e_r$ holds.
\end{enumerate}

Using this characterization of monotone paths, we may $q$-count them as follows. Let $\gamma_0$ be the monotone path corresponding to the sequence $(e_1,\ldots,e_{r-1},a_1,\ldots,a_m,e_r)$. Then
$$\sum_{\gamma}q^{|L_2(\gamma_0,\gamma)|}=(1+q^{m+2})[r-1]!_q\left[\begin{matrix}m+r-1\\m+1\end{matrix}\right]_q.$$

We use condition~\ref{enum_mps_2} to distinguish two cases for a monotone path $\gamma$. If $a_1\prec\cdots\prec a_m\prec e_r$, then $\gamma$ may be constructed by taking an arbitrary permutation $\tau$ of $[r-1]$ and shuffling $(e_{\tau(2)},\ldots,e_{\tau(r-1)})$ with $(a_1,\ldots,a_m,e_r)$. We may identify this shuffle with a partition $\lambda$ in a $(r-2)\times(m+1)$ box. Then $|L_2(\gamma_0,\gamma)|$ is equal to the number of inversions of $\tau$ plus the size of $\lambda$. Hence, these monotone paths contribute
$$[r-1]!_q\left[\begin{matrix}m+r-1\\m+1\end{matrix}\right]_q$$
to the $q$-count. On the other hand, if $e_r\prec a_m\prec\cdots\prec a_1$ holds, then $\gamma$ may be constructed in the same way except that $(e_{\tau(1)},\ldots,e_{\tau(r-2)})$ is shuffled with $(e_r,a_m,\ldots,a_1)$. Flipping $(a_1,\ldots,a_m,e_r)$ and pushing $e_{\tau(r-1)}$ to the end adds a factor of $q^{m+2}$ to the previous $q$-count for these paths.

\begin{remark}
  In Example~\ref{example_zonogon}, we noted the large overlap between zonogons whose rhombic tilings are all coherent and whose discriminantal arrangements are free, as observed in \cite{edelman.reiner:1996free}. In contrast, the arrangements dual to the fiber zonotopes in the $2$-parameter family of zonotopes described in this section are not free in general. This lack of freeness is already present when $r=4$ and $m=1$. Indeed, its characteristic polynomial $(x-1)x(x^2 - 9x + 26)$ does not factor into linear polynomials over $\Zbb$.
\end{remark}

\subsection{Covector axioms}\label{subsec:covector}

Fix a finite set $E$. Recall a \emph{sign vector} is an element of $\{+,-,0\}^E$. Sign vectors form a monoid under composition, where
$$(x\circ y)(e)=\begin{cases}x(e) &\mbox{ if } x(e)\neq 0\\y(e) &\mbox{ otherwise}\end{cases},$$
for $x,y\in\{+,-,0\}^E$. We note that the identity element of the monoid is the 0-vector. The set $\{+,-,0\}$ is partially ordered where $0<+,\ 0<-$, and $+$ and $-$ are incomparable. This ordering extends to $\{+,-,0\}^E$ by the product ordering. The \emph{negation} $-x$ is the vector whose signs are reversed from $x$. Given sign vectors $x,y$, the \emph{separation set} $S(x,y)$ is the set of elements $e\in E$ such that $x(e)=-y(e)$ and $x(e)\neq 0$. The \emph{support} of a sign vector $x$ is the set $\supp(x)=\{e\in E:\ x(e)\neq 0\}$.

A submonoid $\Lcal$ of sign vectors is said to be the set of \emph{covectors of an oriented matroid} if it is closed under negation and
\begin{itemize}
\item\label{axiom} for $x,y\in\Lcal,\ e\in S(x,y)$, there exists $z\in\Lcal$ such that $z(e)=0$ and $z(f)=(x\circ y)(f)=(y\circ x)(f)$ for $f\in E\setm S(x,y)$. (\emph{elimination})
\end{itemize}

As an example, the set of covectors of a hyperplane arrangement defined in \S\ref{subsec_hyparr} satisfies these properties. An oriented matroid that comes from a hyperplane arrangement is said to be \emph{realizable}. While not every oriented matroid is realizable, a theorem of Folkman and Lawrence states that oriented matroids are ``close'' to being realizable \cite{folkman.lawrence:oriented}. For our purposes, we use the following consequence: the proper part of the poset of covectors of any oriented matroid is isomorphic to the poset of faces of a regular CW-sphere.

The set of covectors completely determines the rest of the data associated with an oriented matroid. Hence, we say that two oriented matroids are equal if they have the same set of covectors. For a discussion of other axiomatizations of oriented matroids, we refer to \cite[Chapter 3]{bjorner.lasVergnas.ea:oriented}.

Cellular strings may be defined at the level of oriented matroids as they were for zonotopes in Section~\ref{subsec_cellular_strings}. If $(\Mcal,E)$ is an acyclic oriented matroid with $c_0=-^E$ being the all-negative covector, then a cellular string of $\Mcal$ is a sequence of covectors $(x_1,\ldots,x_l)$ such that $x_1\circ c_0=c_0,\ x_l\circ(-c_0)=-c_0$ and $x_i\circ(-c_0)=x_{i+1}\circ c_0$ for all $i$. We let $\omega(\Mcal)$ denote the set of cellular strings of $\Mcal$.

The main difference between cellular strings of zonotopes and oriented matroids is that we do not know of a notion of ``coherence'' for oriented matroids. In particular, given two realizations $(Z,\pi),(Z^{\pr},\pi^{\pr})$ of an acyclic oriented matroid, a cellular string may be coherent for $(Z,\pi)$ but incoherent for $(Z^{\pr},\pi^{\pr})$. However, using Theorem~\ref{thm_all_coh}, we can say $\Mcal$ is \emph{all-coherent} if $\omega(\Mcal)$ is homeomorphic to a sphere.

\section{All-coherence property for zonotopes}\label{sec_coherence_lemma}

\subsection{Coherence of monotone paths}\label{subsec_lem_galleries}

The main result in this section is Lemma~\ref{lem_coh_galleries}, which says that if all monotone paths of a zonotope are coherent, then so is every cellular string. This lemma appeared previously as \cite[Lemma 4.16]{edman:2015diameter}. 

For this section, we fix a zonotope $Z=Z(A)$ for some matrix $A$ and generic linear functional $\pi$. We let $E$ denote the set of column vectors of $A$. We will assume that $E$ contains no zero vectors nor any parallel or antiparallel pairs of vectors. Furthermore, we scale the elements of $E$ so that $\pi(e)=1$ for each $e\in E$. This results in no loss of generality as the poset of cellular strings is invariant under these changes.

\begin{lemma}[\cite{edman:2015diameter}]\label{lem_coh_galleries}
Let $\Fbf$ be a cellular string. If every cellular string properly refining $\Fbf$ is coherent, then $\Fbf$ is coherent.
\end{lemma}

\begin{proof}

Assume that every cellular string properly refining $\Fbf$ is coherent. Let $(E_1,\ldots,E_l)$ be the ordered set partition of $E$ induced by $\Fbf$. Let $\Fbf^{(1)}$ be a maximal proper cellular string of $\Fbf$. Then $\Fbf^{(1)}$ corresponds to the ordered set partition
$$(E_1,\ldots,E_{k-1},E_1^{\pr},\ldots,E_m^{\pr},E_{k+1},\ldots,E_l)$$
which differs from $\Fbf$ by breaking up a unique block $E_k$. Let $F$ be the element of $\Fbf$ generated by $E_k$. Then $F$ is a zonotope containing a cellular string generated by $(E_1^{\pr},\ldots,E_m^{\pr})$. Consequently, $F$ admits an opposite cellular string generated by $(E_m^{\pr},\ldots,E_1^{\pr})$. Hence, there is a cellular string $\Fbf^{(2)}$ of $Z$ generated by
$$(E_1,\ldots,E_{k-1},E_m^{\pr},\ldots,E_1^{\pr},E_{k+1},\ldots,E_l).$$

By assumption, both $\Fbf^{(1)}$ and $\Fbf^{(2)}$ are coherent cellular strings. Hence, there exist linear functionals $\psi^{(1)},\psi^{(2)}$ such that $\Fbf^{(i)}=\Delta^{\psi^{(i)}}$ for $i=1,2$. By Lemma~\ref{lem_coh_cell_str} and the normalization assumption ($\pi(e)=1$ for all $e\in E$), there is a unique value $\psi^{(i)}_{E_j}=\psi^{(i)}(e)$ for $e\in E_j$. We may similarly define $\psi^{(i)}_{E_j^{\pr}}$ for blocks $E_j^{\pr}$.

Let $t$ be the smallest number in the interval $[0,1]$ such that there exists $j\in\{1,\ldots,m-1\}$ with

$$(1-t)\psi^{(1)}_{E_j^{\pr}}+t\psi^{(2)}_{E_j^{\pr}}=(1-t)\psi^{(1)}_{E_{j+1}^{\pr}}+t\psi^{(2)}_{E_{j+1}^{\pr}}.$$

Let $\psi=(1-t)\psi^{(1)}+t\psi^{(2)}$. It is clear that $\psi$ is constant on each block of $\Fbf^{(1)}$. As before, we let $\psi_{E_j}=\psi(e)$ for $e\in E_j$ and $\psi_{E_j^{\pr}}=\psi(e)$ for $e\in E_j^{\pr}$. Then for $e\in E_k$, we have

$$\psi_{E_1}<\cdots<\psi_{E_{k-1}}<\psi_{E_1^{\pr}}\leq\cdots\leq\psi_{E_m^{\pr}}<\psi_{E_{k+1}}<\cdots<\psi_{E_l}$$

and $\psi_{E_j^{\pr}}=\psi_{E_{j+1}^{\pr}}$. Hence $\Delta^{\psi}$ is a cellular string that refines $\Fbf$ and is properly refined by $\Fbf^{(1)}$. By the maximality of $\Fbf^{(1)}$, we deduce $\Fbf=\Delta^{\psi}$, as desired.

\end{proof}

Letting $\Fbf$ be the trivial cellular string, we deduce the following corollary.

\begin{corollary}\label{cor_mono_path}
If every monotone path is coherent, then $(Z,\pi)$ is all-coherent.
\end{corollary}

\begin{remark}

The analogue of Lemma~\ref{lem_coh_galleries} for general polytope projections fails to hold. For example, there exists a point configuration with two coherent triangulations connected by an incoherent bistellar flip \cite[Example 5.3.4]{deLoera.rambau.santos:2010:TSA:1952022}.

\end{remark}

\begin{figure}

  \centering  
  \includegraphics{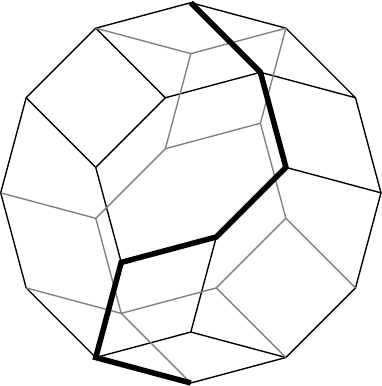}
  \caption{\label{fig_exceptional}A zonotope with the all-coherence property}
  
\end{figure}

\begin{example}
  Consider the zonotope in Figure~\ref{fig_exceptional} with a linear functional $\pi$ such that the bolded path is monotone. This path can be flipped to the left through the middle hexagon face or to the right through one of the rhombic faces. Indeed, one can check that every monotone path is adjacent to exactly two monotone paths. Hence, the graph of monotone paths is a cycle. Since the subgraph of coherent paths is also a cycle, we conclude that every monotone path is coherent. Corollary~\ref{cor_mono_path} then implies that the zonotope is all-coherent.
\end{example}

\subsection{Gradedness}

The face lattice of a polytope is graded by dimension. Thus, if $(\Acal,c_0)$ is all-coherent, then $\omega(\Acal,c_0)$ is a graded lattice. In practice, the poset of cellular strings is typically not graded, which makes this a useful criterion for the all-coherence property. In the following proposition, we give a useful refinement to this criterion.

\begin{proposition}\label{prop_graded}
Let $(\Acal,c_0)$ be all-coherent, and let $r=\rk(\Acal)$.
\begin{enumerate}
\item\label{prop_graded_1} $\omega(\Acal,c_0)$ is graded and every maximal chain contains exactly $r$ elements.
\item\label{prop_graded_2} There does not exist a proper cellular string $\Fbf$ such that $\sum_{x\in\Fbf}(\rk x-1)\geq r-1$.
\end{enumerate}
\end{proposition}

\begin{proof}
The first claim is immediate from Theorem~\ref{thm_all_coh}.

Let $\Fbf$ be a cellular string of $(\Acal,c_0)$, and suppose it contains a covector $x$ such that $\rk x>1$. Set $X=\supp x$. Let $y$ be any covector such that $x\leq y$ and $\rk x=\rk y+1$. Then $y_X$ is a cocircuit of $\Acal_X$, and there exists a cellular string $\Fbf^{\pr}$ of $(\Acal_X,(c_0)_X)$ containing $y_X$. This cellular string lifts to a partial cellular string of $\Acal$ between $x\circ c_0$ and $x\circ(-c_0)$. Replacing $x$ in $\Fbf$ by this partial cellular string, we obtain a new cellular string $\Fbf^{\pr}<\Fbf$ such that
$$1+\sum_{z\in\Fbf^{\pr}}(\rk z-1)\geq\sum_{z\in\Fbf}(\rk z-1).$$

Suppose $(\Acal,c_0)$ has a proper cellular string $\Fbf$ such that $\sum_{x\in\Fbf}(\rk x-1)\geq r-1$. Then we may apply the above construction repeatedly to obtain a chain with at least $r$ elements. Extending this chain by the trivial cellular string gives a chain of with at least $r+1$ elements, contradicting (\ref{prop_graded_1}).
\end{proof}


Proposition~\ref{prop_graded}(\ref{prop_graded_2}) gives a useful test for incoherence, as in the following example.

\begin{example}\label{example_rank3}
Let $Z$ be the image of the $4$-cube under the map
$$\pi=\left(\begin{matrix}1 & 0 & 1 & 0\\0 & 1 & 1 & 0\\0 & 1 & 0 & 1\end{matrix}\right)$$
and define $f:\Rbb^3\ra\Rbb$ such that $f(x,y,z)=x+y+z$. Let $\Mcal$ be the acyclic oriented matroid associated to $(Z,f)$ with ground set $\{a,b,c,d\}$ corresponding to the six columns of $\pi$. We order the ground set $a<b<c<d$ to agree with the ordering of the columns of $\pi$. Then $\Mcal$ has a unique circuit $(+,+,-,-)$, along with its negative $(-,-,+,+)$. The other $14$ elements of $\{+,-\}^{\{a,b,c,d\}}$ are maximal covectors of $\Mcal$.

The oriented matroid $\Mcal$ has a cellular string consisting of two cocircuits $(0,-,-,0)$ and $(+,0,0,+)$. But $(\rk(0,-,-,0) - 1) + (\rk(+,0,0,+) - 1) = 2 = \rk\Mcal-1$, so $(Z,f)$ is not all-coherent by Proposition~\ref{prop_graded}(\ref{prop_graded_2}).
\end{example}

\subsection{Matroid operations}\label{subsec_operations}

Let $V$ be a real vector space of dimension $r$. Fix an $r\times n$ real matrix $A$, and let $E$ be the set of column vectors of $A$. We let $Z=Z(A)$ and let $\pi$ be a linear functional on $\Rbb^r$ that is generic with respect to $Z$. This defines an acyclic oriented matroid $\Mcal$. Since the coherence of cellular strings is invariant under scaling the columns of $A$, we will assume throughout this section that $\pi(e)=1$ for all $e\in E$. Let $\Acal$ be the set of hyperplanes in $V^*$ for which $\Lcal(\Acal)$ is the normal fan of $Z$. Let $c_0$ be the chamber of $\Acal$ containing $\pi$. For simplicity, we will assume that the set $E$ does not contain any zero vectors or pairs of parallel vectors.

\begin{lemma}[cf. \cite{edman:2015diameter}(Corollary 5.2)]\label{lem_coloop}
Let $\wtil{Z}$ be the direct sum of $Z$ with an $m$-cube. We realize this zonotope by a matrix $\wtil{A}=A\oplus I_m$ where $I_m$ is the rank $m$ identity matrix. If $\wtil{\pi}$ is a generic linear functional on $\Rbb^{r+m}$ whose restriction to $Z$ is equal to $\pi$, then $(Z,\pi)$ is all-coherent if and only if $(\wtil{Z},\wtil{\pi})$ is all-coherent.
\end{lemma}

\begin{proof}
It suffices to prove the lemma for $m=1$ since if $m>1$, the zonotope $\wtil{Z}$ decomposes as $(Z\oplus[0,1]^{m-1})\oplus[0,1]$, and the result follows by induction. Hence, we will assume $m=1$ for the remainder of the proof.

Let $e$ be the unit vector in the last coordinate of $\Rbb^{r+1}$. Let $\wtil{\pi}$ be a generic linear functional on $\Rbb^{r+1}$ whose restriction to $Z$ is $\pi$. We may assume that $\wtil{\pi}(e)=1$ without affecting coherence of a cellular string on $\wtil{Z}$.

Let $\wtil{\gamma}$ be a cellular string on $\wtil{Z}$. Let $\gamma$ be the cellular string on $Z$ obtained by deleting $e$ from $\wtil{\gamma}$. Suppose $\gamma$ is $\pi$-coherent. Then there exists a linear functional $\psi$ on $\Rbb^r$ that picks out $\gamma$. Let $(E_1,\ldots,E_l)$ be the ordered set partition of $\{a_1,\ldots,a_n\}$ induced by $\gamma$. This means for $a\in E_i,\ b\in E_j$, $\psi(a)\leq\psi(b)$ if $i\leq j$.

We define a linear functional $\wtil{\psi}$ that determines the string $\wtil{\gamma}$ as follows. For each $a_i$, set $\wtil{\psi}(a_i)=\psi(a_i)$. If $e$ is added to a block $E_i$ to form $\wtil{\gamma}$, then set $\wtil{\psi}(e)=\psi(a)$ for some $a\in E_i$. If $e$ is in its own block between $E_i$ and $E_{i+1}$, then set $\wtil{\psi}(e)$ to be a value strictly between $\psi(a)$ and $\psi(b)$ where $a\in E_i,\ b\in E_{i+1}$. In either case, we have produced the desired linear functional.

Conversely, if we start with a $\wtil{\pi}$-coherent string $\wtil{\gamma}$ induced by $\wtil{\psi}$, then $\psi$ induces $\gamma$ on $Z$. This completes the proof.
\end{proof}

For oriented matroids $(\Mcal,E),(\Mcal^{\pr},E^{\pr})$ on disjoint ground sets $E,E^{\pr}$, the \emph{direct sum} $\Mcal\oplus\Mcal^{\pr}$ of $\Mcal$ and $\Mcal^{\pr}$ is an oriented matroid with covectors $\Lcal(\Mcal)\times\Lcal(\Mcal^{\pr})$. An oriented matroid is \emph{indecomposable} if cannot be expressed as a direct sum of two nontrivial oriented matroids. It is straightforward to show that the circuits of $\Mcal\oplus\Mcal^{\pr}$ is the union of $\Ccal(\Mcal)\times\{0\}$ and $\{0\}\times\Ccal(\Mcal^{\pr})$ \cite[Proposition 7.6.1]{bjorner.lasVergnas.ea:oriented}.

\begin{lemma}\label{lem_decomposable}
Let $\Mcal$ be an acyclic oriented matroid that decomposes as $\Mcal_1\oplus\Mcal_2$. If both $\Mcal_1$ and $\Mcal_2$ have a circuit supported by at least 3 elements, then $\Mcal$ is not all-coherent.
\end{lemma}

\begin{proof}
Let $E_i$ be the ground set of $\Mcal_i$ for $i=1,2$. Assume that $\Mcal_1$ and $\Mcal_2$ each have a circuit supported by at least 3 elements. Then for $i=1,2$ there exists $x_i\in\Lcal(\Mcal_i)$ such that $x_i$ is a corank 1 covector and $\rk\Mcal_i\setm x_i^{-1}(0)\geq 2$.

Let $\Fbf_i=(y_1^{(i)},y_2^{(i)},\ldots,y_{l_i}^{(i)})$ be a cellular string of $\Mcal_i$ containing $x_i$ such that each cell $y_j^{(i)}$ is rank 1 unless $y_j^{(i)}=x_i$ for $i=1,2$. Without loss of generality we may assume that $l_1\leq l_2$. By the assumption on $x_i$, we have $l_1\geq 3$. We define a new string $\Fbf=(y_1,\ldots,y_m)$ for $\Mcal$ where $y_j=y_j^{(1)}\oplus y_j^{(2)}$ if $j\leq l_1$ and $y_j=+^{E_1}\oplus y_j^{(2)}$ if $l_1<j\leq l_2$. Then

\begin{align*}
\sum(\rk y_j-1) &= \sum_{j=1}^{l_1}\rk y_j^{(1)}+\sum_{j=1}^{l_2}(\rk y_j^{(2)}-1)\\
&\geq (\rk\Mcal_1+1)+(\rk\Mcal_2-1)\\
&=\rk\Mcal.
\end{align*}

By Proposition~\ref{prop_graded}, we conclude that $\Mcal$ is not all-coherent.
\end{proof}

\begin{example}\label{example_rank4}
Let $Z$ be the image of the $6$-cube under the map
$$\left(\begin{matrix}1 & 0 & .5 & 0 & 0 & 0\\0 & 1 & .5 & 0 & 0 & 0\\0 & 0 & 0 & 1 & 0 & .5\\0 & 0 & 0 & 0 & 1 & .5\end{matrix}\right)$$
and define $\pi:\Rbb^4\ra\Rbb$ such that $\pi(w,x,y,z)=w+x+y+z$. We note that $\pi(v)=1$ for each column vector $v$. Let $\Mcal$ be the acyclic oriented matroid associated to $(Z,\pi)$ with ground set $\{a,b,c,d,e,f\}$ corresponding to the four columns of the matrix. Then $\Mcal$ decomposes into two rank 2 matroids, each of corank 1. Hence, $(Z,\pi)$ is not all-coherent.

In particular, we claim that the cellular string $\{0--0--,\ +--+-0,\ +-0+-+,\ +0++0+\}$ is not coherent. Suppose to the contrary that it is picked out by a linear functional $\psi:\Rbb^4\ra\Rbb$. By definition, this means
$$\psi(a)=\psi(d)<\psi(f)<\psi(c)<\psi(b)=\psi(e).$$
But the first and last equality imply
$$2\psi(c) = \psi(a)+\psi(b) = \psi(d)+\psi(e)=2\psi(f),$$
contradicting the inequality $\psi(f)<\psi(c)$.
\end{example}


The last general construction we consider are weak map images. Given two oriented matroids $\Mcal,\Mcal^{\pr}$ with the same ground set $E$, there is a \emph{weak map} $\Mcal\rsqa\Mcal^{\pr}$ if for every circuit $x$ of $\Mcal$, there exists a circuit $y$ of $\Mcal^{\pr}$ such that $y\leq x$. That is, every circuit of $\Mcal$ ``contains'' a circuit of $\Mcal^{\pr}$. Geometrically, a weak map image of a vector configuration is obtained by moving the vectors to a more special position. The map is said to be \emph{rank-preserving} if $\Mcal$ and $\Mcal^{\pr}$ have the same rank. We remark that if $e\in E$, there exist weak maps $\Mcal\rsqa\Mcal/e$ and $\Mcal\rsqa\Mcal\setm e$. Here, we view the contraction $\Mcal/e$ as an oriented matroid on $E$ where $e$ is a coloop, and we view the deletion $\Mcal\setm e$ as an oriented matroid on $E$ where $e$ is a loop. In this way, the weak maps $\Mcal\rsqa\Mcal/e$ and $\Mcal\rsqa\Mcal\setm e$ are both rank-preserving.

\begin{proposition}\label{prop_weak}
Let $\Mcal,\Mcal^{\pr}$ be acyclic oriented matroids with the same rank. Assume that there is a rank-preserving weak map $\Mcal\rsqa\Mcal^{\pr}$. If $\Mcal$ is all-coherent, then so is $\Mcal^{\pr}$.
\end{proposition}

To prove this proposition, one may observe that the list of oriented matroids in the classification in Section~\ref{sec_classification} is closed under rank-preserving weak map images. Unfortunately, we do not have a direct proof.

The rank preservation assumption in Proposition~\ref{prop_weak} is necessary since the $n$-cube $I^n$ is all-coherent, but every oriented matroid on $E=\{1,\ldots,n\}$ is a weak map image of the free oriented matroid $\Mcal(I^n)$, including incoherent ones.

\section{Classification}\label{sec_classification}

\subsection{Affine models and monotone families}

Let $Z$ be a zonotope generated by $v_1$, \dots, $v_n \in V$ and let $\pi \in V^\ast$ be a generic linear functional. Let $H_\pi$ be the hyperplane $\{ x \in V : \pi(x) = 1 \}$. We define
\[
\Pcal(Z,\pi) := \bigcup_{i=1}^n (\{ tv_i : t \in \Rbb \} \cap H_\pi).
\]
Thus, $\Pcal(Z,\pi)$ is a set of $n$ points in $H_\pi$. Note that the oriented matroid of $Z$ depends only on the oriented matroid of $\Pcal(Z,\pi)$ as an affine point set. In particular, the all-coherence property depends only on this oriented matroid. We will call any set of points in affine space with the same oriented matroid as $\Pcal(Z,\pi)$ an \emph{affine model} for $(Z,\pi)$.

Given an affine space $W$, we define a \emph{(closed) half-space} of $W$ to be a subset of the form $\{x \in W : \psi (x) \le c \}$ for some affine map $\psi : W \to \Rbb$ and $c \in \Rbb$. Note that $\emptyset$ and $W$ are half-spaces of $W$ in this definition. We denote the set of all half-spaces of $W$ by $HS(W)$, and give it a topology in the obvious way. For $K \in HS(W)$, we denote by $\partial K$ and $K^\circ$ the boundary and interior, respectively, of $K$ as a subspace of $W$.

If $\Pcal$ is a finite set of points in $W$, we define a \emph{monotone family} with respect to $\Pcal$ to be a continuous path $\{K_t\}_{t \in [0,1]}$ in $HS(W)$ such that
\begin{enumerate}
\item $K_0 = \emptyset$
\item $K_1 = W$
\item For every $p \in \Pcal$, there exists $t_p \in (0,1)$ such that $p \in \partial K_{t_p}$ and $p \in K_{t}^\circ$ for all $t > t_p$.
\end{enumerate} 
To each monotone family $\{K_t\}_{t \in [0,1]}$ with respect to $\Pcal$ we associate an ordered set partition of $\Pcal$ as follows: Let $t_1 < \dotsb < t_l$ be all the numbers in $(0,1)$ such that $\partial K_{t_i} \cap \Pcal \neq \emptyset$. Then the ordered set partition associated to $\{K_t\}_{t \in [0,1]}$ is
\[
(\partial K_{t_1} \cap \Pcal, \partial K_{t_2} \cap \Pcal, \dotsc, \partial K_{t_l} \cap \Pcal).
\]
We call any partition of $\Pcal$ obtained in this way a \emph{monotone partition}. As the following proposition states, monotone partitions of affine models are in bijection with proper cellular strings.

\begin{proposition} \label{monoparttostring}
Let $\Pcal$ be an affine model for $(Z,\pi)$. Then there is a bijection from the set of all monotone partitions of $\Pcal$ to the set of all proper cellular strings of $(Z,\pi)$, and such that the image of $(E_1,\dotsc,E_l)$ under this bijection is $(F_1,\dotsc,F_l)$ where $\dim(F_i) = \lvert E_i \rvert$.
\end{proposition}

\begin{proof}
We may assume $\Pcal = \Pcal(Z,\pi)$. We have a homeomorphism from $S(V^\ast) := \{ \psi \in V^\ast : \lVert \psi \rVert = 1\}$ to $HS(H_\pi)$ given by $\psi \mapsto \{ x \in H_\pi : \psi(x) \ge 0 \}$. Recall that each proper cellular string of $(Z,\pi)$ can be specified by a path $\gamma : [0,1] \to S(V^\ast)$ such that $\gamma(0) = -\pi/\Vert \pi \rVert$, $\gamma(1) = \pi/\Vert \pi \rVert$, and for each $i$, $\gamma$ intersects the hyperplane $\{ \psi \in V^\ast : \psi(v_i) = 0 \}$ exactly once. The image of the set of all such paths under the above homeomorphism is precisely the set of monotone families with respect to $\Pcal(Z,\pi)$. It is easy to check that this homeomorphism induces a bijection between the proper cellular strings of $(Z,\pi)$ and the monotone partitions of $\Pcal(Z,\pi)$, and this bijection satisfies the desired property.
\end{proof}

In the proof of classification we will need the following lemma.

\begin{lemma} \label{monofamconstruction}
Let $\Pcal$ be a finite set of points in an affine space $W$. Suppose that $K_1$, \dots, $K_m \in HS(W)$ are such that for each $1 \le i < m$, $(K_i \setminus K_{i+1}^\circ) \cap \Pcal = \emptyset$. Then there exists a monotone family with respect to $\Pcal$ containing $K_1$, \dots, $K_m$.
\end{lemma}

\begin{proof}
We define a family $\{L_t\}_{t \in (-\infty,\infty)}$ of half-spaces of $W$ as follows: For $i = 1$, \dots, $m$, let $L_i = K_i$. If $K_1 = \{x \in W : \psi_1 (x) \le c_1 \}$, then for $t < 1$ define
\[
L_t = \{ x \in W : \psi_1(x) \le c_1 + t - 1 \}
\]
and if $K_m = \{x \in W : \psi_m (x) \le c_m \}$, for $t > m$ define
\[
L_t = \{ x \in W : \psi_m(x) \le c_m + t - m \}.
\]
Finally, if $i < t < i+1$ for $i=1$, \dots, $m-1$, we define $L_t$ as follows. Let $\ell_i = \partial K_i \cap \partial K_{i+1}$. Then $\ell_i$ is a codimension 2 affine subspace of $W$, and the set
\[
\Phi_i := \{ K \in HS(W) : \partial K \supset \ell_i \}
\]
is homeomorphic to a circle and contains $K_i$, $K_{i+1}$. We let $\{ L_t \}_{t \in [i,i+1]}$ be the shortest arc in $\Phi_i$ with $L_i = K_i$ and $L_{i+1} = K_{i+1}$.

The condition $(K_i \setminus K_{i+1}^\circ) \cap \Pcal = \emptyset$ guarantees that in the above construction, for every $p \in \Pcal$ there exists $t_p \in (-\infty,\infty)$ such that $p \in \partial L_{t_p}$ and $p \in L_t^\circ$ for all $t > t_p$. Thus, setting $L_{-\infty} = \emptyset$ and $L_{\infty} = W$ we obtain a monotone family with respect to $\Pcal$ containing $K_1$, \dots, $K_m$.
\end{proof}

\subsection{Classification theorem}

To classify the all-coherent zonotopes, it suffices to give their affine models. Here are four such affine models:

\begin{enumerate}
\item The set $U_{2,n}$ of $n$ points in a straight line. This corresponds to a zonotope of dimension 2 and any generic linear functional.
\item For $r \ge 3$, the set $E_{r,m} = \{e_1,\dots,e_r,a_1,\dots,a_m\}$, where $e_1$, \dots, $e_r$ are the vertices of an $(r-1)$-dimensional simplex and $a_1$, \dots, $a_m$ are points in the interior of the simplex such that $a_1$, \dots, $a_m$, $e_r$ are collinear. This is the affine model for the all-coherent family in Section~\ref{subsec_allcoh_family}.
\item For $r \ge 3$, the set $\tilde{E}_{r,m}$, which is defined in the same way as $E_{r,m}$ except $a_1$ is on the boundary of the simplex.
\item The set $R_3$, the set of 6 points in $\Rbb^2$ shown in Figure~\ref{R3pic}. The corresponding zonotope $Z$ for $R_3$ is shown in Figure~\ref{fig_exceptional}, with $\pi$ the vertical direction in the figure. One can check that the poset of proper cellular strings of $(Z,\pi)$ is homoemorphic to a circle, which implies the all-coherence property for $R_3$.

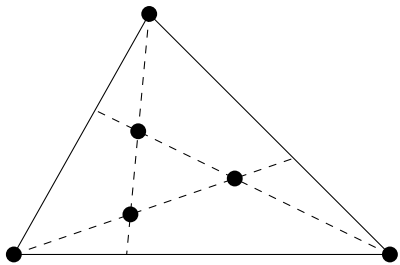
\begin{figure}
\begin{centering}
\begin{tikzpicture}[scale=2]
	\coordinate (A) at (0,0);
	\coordinate (B) at (.9,1.6);
	\coordinate (C) at (2.5,0);
	\coordinate (A') at ($.4*(B) + .6*(C)$);
	\coordinate (B') at ($.3*(C) + .7*(A)$);
	\coordinate (C') at ($.4*(A) + .6*(B)$);
	\draw (A) -- (B) -- (C)--cycle;
	\draw[dashed, name path=(AA')] (A) -- (A');
	\draw[dashed, name path=(BB')] (B) -- (B');
	\draw[dashed, name path=(CC')] (C) -- (C');
	\path[name intersections={of=(AA') and (BB'), by=D}];
	\path[name intersections={of=(BB') and (CC'), by=E}];
	\path[name intersections={of=(CC') and (AA'), by=F}];
	\foreach \point in {A,B,C,D,E,F}
		\fill[black] (\point) circle (1.5pt);
\end{tikzpicture}
\end{centering}
\caption{\label{R3pic} The affine model $R_3$}
\end{figure} 
\end{enumerate}

We now show that, up to the addition of coloops, these are the only all-coherent affine models.

\begin{theorem}
Let $Z$ be a zonotope of dimension $r$ and $\pi$ a generic linear functional. The following are equivalent.
\begin{enumerate}
\item $(Z,\pi)$ is all-coherent.
\item There does not exist a proper cellular string $(F_1,\dots,F_l)$ of $(Z,\pi)$ with
\[
\sum_{i=1}^l (\dim(F_i)-1) \ge r-1.
\]
\item The affine model for $(Z,\pi)$ is $E \cup L$, where $E$ is of the form $\emptyset$, $U_{2,n}$, $E_{r,m}$, $\tilde{E}_{r,m}$, or $R_3$ and $L$ is a (possibly empty) set of coloops of $E \cup L$.
\end{enumerate}

\begin{proof}
(1)$\Ra$(2) follows from Proposition~\ref{prop_graded}, and (3)$\Ra$(1) has been proven earlier. Our goal is to prove (2)$\Ra$(3).

Assume (2) holds. Let $\Pcal$ be an affine model for $(Z,\pi)$. We may assume $\Pcal$ is full-dimensional in $\Rbb^{r-1}$. For any polytope $Q$, let $V(Q)$ denote its vertex set. We first show the following.

\begin{claim}
The convex hull of $\Pcal$ is a simplex.
\end{claim}

\begin{proof}
Suppose the contrary, and let $Q$ be the convex hull of $\Pcal$. Let $F$ be any facet of $Q$. If there are at least two vertices of $Q$ not in $F$, then there is an edge $e$ of $Q$ disjoint from $F$. Otherwise, $Q$ is a pyramid over $F$, and since $Q$ is not a simplex, if we replace $F$ with any other facet of $Q$ then there will be two vertices of $Q$ not in $F$. So we may assume there is an edge $e$ of $Q$ disjoint from $F$.

Let $K_1$ be the half-space of $\Rbb^{r-1}$ which does not contain $Q$ and whose boundary supports $F$. Let $K_2$ be any half-space of $\Rbb^{r-1}$ which contains $Q$ and whose boundary supports $e$. Then $(K_1 \setminus K_2^\circ) \cap Q = \emptyset$, so by Lemma~\ref{monofamconstruction}, there is a monotone family with respect to $\Pcal$ containing $K_1$, $K_2$. In particular, the associated monotone partition of $\Pcal$ contains $V(F)$ and $V(e)$. Thus, by Proposition~\ref{monoparttostring}, there is a cellular string $(F_1,\dots,F_l)$ of $(Z,\pi)$ with
\[
\sum_{i=1}^l (\dim(F_i)-1) \ge (\lvert V(F) \rvert -1) + (\lvert V(e) \rvert - 1) \ge (r-2) + 1 = r-1.
\]
This contradicts (2), proving the claim.
\end{proof}

Now, let $\sigma$ denote the convex hull of $\Pcal$.

\begin{claim} \label{vertexonline}
If $a_1$, $a_2$ are distinct points in $\Pcal \setminus V(\sigma)$, then there exists $b \in V(\sigma)$ such that $a_1$, $a_2$, and $b$ are collinear. 
\end{claim}

\begin{proof}
Suppose the contrary. Let $\ell$ be the line through $a_1$ and $a_2$, and let $p_1$, $p_2$ be the endpoints of the segment $\ell \cap \sigma$. Let $C_1$ and $C_2$ be the vertex sets of the minimal faces of $\sigma$ containing $p_1$ and $p_2$, respectively. By assumption, $C_1$ and $C_2$ have at least two elements. It is thus possible to choose a partition $(B_1,B_2)$ of $B$ such that $C_i \not\subseteq B_j$ for all $i$, $j$.

Let $\sigma_1$, $\sigma_2$ be the faces of $\sigma$ with vertex sets $B_1$, $B_2$ respectively. Let $K_1$ be a half-space of $\Rbb^{r-1}$ which does not contain $\sigma$ and whose boundary supports $\sigma_1$. Let $K_2$ be a half-space of $\Rbb^{r-1}$ whose boundary contains $\ell$ and such that $\sigma_1 \subset K_2^\circ$ and $\sigma_2 \cap K_2 = \emptyset$. (This is possible by the definitions of $\sigma_1$ and $\sigma_2$.) Finally, let $K_3$ be a half-space of $\Rbb^{r-1}$ which contains $\sigma$ and whose boundary supports $\sigma_3$. Then $(K_1 \setminus K_2^\circ) \cap \sigma = (K_2 \setminus K_3^\circ) \cap \sigma = \emptyset$, so by Lemma~\ref{monofamconstruction} there is a monotone family with respect to $\Pcal$ containing $K_1$, $K_2$, $K_3$. Hence, there is a cellular string $(F_1,\dots,F_l)$ of $(Z,\pi)$ with
\[
\sum_{i=1}^l (\dim(F_i)-1) \ge (\lvert V(\sigma_1) \rvert - 1) + (\lvert \{a_1,a_2\} \rvert - 1) + (\lvert V(\sigma_2) \rvert - 1) = r-1
\]
which contradicts (2).
\end{proof}

We can now complete the proof. If $\Pcal \setminus V(\sigma)$ has at most one points, then it is easy to see that $\Pcal$ is of one of the forms described in (3). Assume $\Pcal \setminus V(\sigma)$ has at least two distinct points $a_1$, $a_2$, and let $\ell_{12}$ be the line through these points. By Proposition~\ref{vertexonline}, $\ell_{12}$ contains some vertex $b_{12}$ of $\sigma$. If every point of $\Pcal \setminus V(\sigma)$ is on $\ell_{12}$, then $\Pcal$ is of one of the forms described in (3), as desired.

Assume there is some $a_3 \in \Pcal \setminus V(\sigma)$ not on $\ell_{12}$. Then by Proposition~\ref{vertexonline}, there are vertices $b_{13}$ and $b_{23}$ of $\sigma$ such that $a_1$, $a_3$, $b_{13}$ lie on a line $\ell_{13}$ and $a_2$, $a_3$, $b_{23}$ lie on a line $\ell_{23}$. Moreover, since $a_3 \notin \ell_{12}$, $b_{12} \neq b_{13} \neq b_{12}$. Since $\ell_{12}$ and $\ell_{13}$ intersect at $a_1$, there is a 2-plane $P$ containing $\ell_{12}$ and $\ell_{13}$. Since $a_2$, $a_3 \in P$, we have $\ell_{23} \subset P$, and hence $b_{23} \in P$. Hence $a_1$, $a_2$, $a_3$, $b_{12}$, $b_{13}$, $b_{23}$ all lie in $P$, and thus $a_1$, $a_2$, $a_3$ are all in the face $F$ of $\sigma$ with vertices $b_{12}$, $b_{13}$, $b_{23}$. It is easy to check that these six points must be in the configuration $R_3$.

Now suppose that there is some $a_4 \in \Pcal \setminus (V(\sigma) \cup \{a_1,a_2,a_3\})$. The above argument works with $a_4$ instead of $a_3$. However, because $a_1$ is in the interior of $F$, $F$ is the only 2-face of $\sigma$ containing $a_1$, and hence $a_4 \in F$. However, it is easy to check that one cannot have 4 distinct points in a triangle, not all on the same line, such that the line through any two of them passes through a vertex of the triangle. Thus there is no such $a_4$. Hence, $\Pcal = R_3 \cup L$, where $L$ is a set of coloops of $\Pcal$. This concludes the proof.
\end{proof}
\end{theorem}

\section*{Acknowledgements}

The authors thank Victor Reiner for offering advice at several stages in this project. The second author was supported by NSF/DMS-1351590. The fourth author was partially supported by NSF/DMS-1148634. Part of this material was based on work supported by NSF/DMS-1440140 while the third and fourth authors were in residence at the Mathematical Sciences Research Institute in Fall 2017.

\bibliography{bib_coherent_strings}{}
\bibliographystyle{plain}

\end{document}